%% file: ccsasArev_Feb11.tex
\documentclass[11pt]{amsart}
\usepackage{amsmath,amsthm,amssymb}
\usepackage{graphicx}
  \usepackage{tikz}
  \usetikzlibrary{calc,arrows,angles,quotes}
\numberwithin{equation}{section}
\newcommand{\R}{\mathbb{R}}
\newcommand{\T}{\mathcal{T}}
\newcommand{\TT}{\mathbb{T}}
\newcommand{\NN}{\mathbb{N}}

\newcommand{\V}{\mathcal{V}}
\newcommand{\A}{{\bf{A}}}
\newcommand{\B}{{\bf{B}}}
\newcommand{\vs}{{\bf{s}}}
\newcommand{\vt}{{\bf{t}}}
\newcommand{\vn}{{\bf{n}}}
\newcommand{\ve}{{\bf{e}}}
\newcommand{\vb}{{\bf{b}}}
\newcommand{\vv}{{\bf{v}}}
\newcommand{\vw}{{\bf{w}}}
\newcommand{\vx}{{\bf{x}}}
\newcommand{\vy}{{\bf{y}}}
\newcommand{\vz}{{\bf{z}}}
\newcommand{\vP}{{\bf{P}}}
\newcommand{\vV}{{\bf V}_h}
\newcommand{\vvh}{{\bf v}_h}
\newcommand{\bD}{{\bf{D}}}
\newcommand{\bA}{{\bf{A}}}
\newcommand{\bB}{{\bf{B}}}
\newcommand{\bF}{{\bf{F}}}
\newcommand{\bM}{{\bf{M}}}
\newcommand{\bT}{{\bf{T}}}
\newcommand{\Le}{\mathcal{L}}
\newcommand{\La}{\Lambda}

\DeclareMathOperator{\sspan}{span}
\DeclareMathOperator{\ddiv}{div}

\DeclareMathOperator{\im}{Im}

\newcommand{\operp}{\;{\bigcirc\hspace{-4.5mm}\perp}\;}
\newcommand{\CR}{{\rm CR}}
\renewcommand{\P}{{\rm P}}

\newcommand{\pw}{{\mathrm pw}}

\newcommand{\trinorm}[1]{\ensuremath{|\!|\!|} #1 \ensuremath{|\!|\!|}}

\newtheorem{theorem}{Theorem}
\numberwithin{theorem}{section}
\newtheorem{proposition}[theorem]{Proposition}
\newtheorem{lemma}[theorem]{Lemma}
\newtheorem{definition}{Definition}

\title[Crouzeix-Raviart is Stokes stable]{Crouzeix-Raviart triangular \\ elements   are  inf-sup  stable}
\author{Carsten Carstensen}
\address{Humboldt-Universit\"at zu Berlin,
10099 Berlin, Germany. cc@math.hu-berlin.de}

\author{Stefan A. Sauter}
\address{Institut f\"{u}r Mathematik,
Universit\"{a}t Z\"{u}rich, Winterthurerstr 190, CH-8057 Z\"{u}rich,
Switzerland. stas@math.uzh.ch}

\thanks{}
\begin{document}

\dedicatory{
dedicated to Michel Crouzeix}

\begin{abstract}
The Crouzeix-Raviart triangular finite elements are $\inf$-$\sup$
stable for the Stokes equations for any mesh with at least one interior vertex.
This result affirms a {\em conjecture of Crouzeix-Falk}
from 1989  for $p=3$. Our proof  applies
to {\em any odd degree} $p\ge 3$ and concludes  the overall stability analysis:
Crouzeix-Raviart triangular finite elements of degree  $p$  in two dimensions and the piecewise polynomials of degree $p-1$ with vanishing integral
form a stable Stokes pair {\em for all positive integers} $p$.
\end{abstract}

\maketitle


{\bf AMS.} 65N30,  65N12, 65N15

\section{Introduction}
The Crouzeix-Raviart finite element method was invented to approximate
the solution to the Stokes equations for  polynomial degree
$p\in \NN$ in \cite{CrouzeixRaviart}.
The realization of the nonconforming schemes for a higher polynomial degree $p$ is less obvious in 2D
\cite{Ainsworth_Rankin,Baran_Stoyan,CCSS_CR2,ChaLeeLee,ccss_2012,BaranCVD}
and is still in its infancy in 3D \cite{CDS}. The   case $p=1$ is classical \cite{CrouzeixRaviart} and  the 2D version of $p=2$ is known as the Fortin-Soulie element \cite{Fortin_Soulie}.
The  Crouzeix-Raviart finite element spaces are significantly  different for even and odd polynomial degrees \cite{ccss_2012}.

The important cubic case $p=3$  was the topic of the seminal contribution  \cite{Crouzeix_Falk} that establishes, in particular,  the macro-element technology
for the  proof of the Stokes stability for nonconforming  finite element methods. It also guarantees
$\inf$-$\sup$ stability up to exceptional configurations of four edge-connected triangles
 characterized by  some geometric quantity $D=0$ in \cite[Prop. 3.1]{Crouzeix_Falk}  and ends with a conjecture:
 {\em We conjecture  that } $\CR^p_0(\T;\R^2)\times \P_{p-1}(\T)/\R$ forms {\em  a stable Stokes pair for any triangulation ...
satisfying the minimal angle condition and containing an interior vertex} (quoted, in modified notation, from the Remark in \cite{Crouzeix_Falk}).
Theorem~\ref{thm1}  below in this paper for  $p=3$ resolves this Crouzeix-Falk  conjecture from 1989 and is recycled to cover all the open remaining cases  $p=3,5,7,\dots$
simultaneously; the accompanying paper  \cite{CCSS_CR2} further clarifies the role of the singular functions for $p=5,7,9,\dots$.

The stability of  the Crouzeix-Raviart finite element method in the 2D Stokes equations
for  $p=4,6,8,\dots$ is shown  in  \cite{Baran_Stoyan} with a discussion of singular geometries.

The cubic case is the hardest, because   combinatorial proofs and the concept of singular points  do not fully characterize the
kernel of the piecewise divergence operator (cf. \cite[Rem. 3.1]{ScottVogelius} for proofs)
inside the macro-element methodology \cite{zbMATH03825382,Crouzeix_Falk,Stenberg_macro}.
It is emphasized that singular vertices do {\em not } play any role in this paper unlike, e.g.,
in \cite{Baran_Stoyan,CCSS_CR2,GuzmanScott2019,ScottVogelius}.

\medskip

The remaining parts of this paper are organised as follows. Theorem~\ref{thm1} and its necessary notation follow in Section~2. Section~3 reduces the
analysis to four parameters per triangle, before Section~4 introduces the geometry of a interior vertex patch.  Section~5 determines the coefficient matrix
that equivalently describes the  assertion by ist kernel vectors. The key point follows in Section~6 with the proof that the constant pressure is the only kernel
vector.  The linear algebra of some kernel vectors is postponed to Section~7. Some remarks on the macro-element technique in Section~8  conclude  the proof and the paper.

\section{Setting}
Given a  (shape-) regular triangulation $\T$ of the planar polygonal bounded Lip\-schitz domain $\Omega\subset\R^2$ into triangles
 with set of edges $\mathcal{E}$ (resp. edges on the boundary  $\mathcal{E}(\partial\Omega)$), let  $P_p(\T;\R^2)\equiv P_p(\T)\times P_p(\T)$ denote the piecewise polynomials
of (total) degree at most $p\ge 1$ in two components.
 The nonconforming velocity space
\begin{align*}
\CR^p_0(\T;\R^2)&:=\{ \vv_h\in P_p(\T;\R^2) : \, \vv_h \text{ is continuous at Gauss points of } E\in\mathcal{E} \\
&\text{and vanishes at Gauss points of any boundary edge  }  E\in\mathcal{E}(\partial\Omega) \}
\end{align*}
is considered with the piecewise polynomial pressure space  $\P_{p-1}(\T)/\R$. The pressure space $\P_{p-1}(\T)\cap L^2_0(\Omega) \equiv\P_{p-1}(\T)/\R$
 is the quotient space of the  piecewise polynomials  $\P_{p-1}(\T)$ of (total) degree at most $p-1$ divided by the constants $\R$ and identified with
the representations in $L^2_0(\Omega):=\{ q\in L^2(\Omega):\int_\Omega q\, dx=0\}$ with integral zero.
The Gauss points of an edge $E$ are the affine images  of the zeros in the interval $(-1,1)$
of the $p$-th order Legendre polynomial $\Le_p:[-1,1]\to\R $ normalized by $\Le_p(1)=1$.
The continuity at the Gauss points in the above definition of
the velocity space $\CR^p_0(\T;\R^2)$ along an edge  $E$ is equivalent to the orthogonality of the jump $[\vv_h]_E$ of $\vv_h\in P_p(\T;\R^2)$ onto the
one-dimensional polynomials $P_{p-1}(E;\R^2)$  in $L^2(E;\R^2)$.

The main result in this paper is the Stokes stability:  The piecewise divergence
\begin{equation}\label{eqccintro1}
\ddiv_{\pw}:\CR^p_0(\T;\R^2) \to \P_{p-1}(\T)/\R
\end{equation}
is surjective and allows for  a bounded  linear right-inverse with an $h$-independ\-ent operator norm  $C_p$.
Standard notation on Lebesgue and Sobolev spaces and their (semi-) norms applies throughout the paper.
The pressure space is endowed with the
 $L^2$ norm $ \|\bullet \|_{L^2(\Omega)}$  in $\Omega$ and the nonconforming velocities are endowed with
the piecewise $H^1$ seminorm $\trinorm{\bullet}_\pw$,  defined by
\begin{equation}\label{eqpiecewiseH1}
\trinorm{\vv_h}_\pw=\sqrt{\textstyle \sum_{T\in\T} | \vv_h|_{H^1(\text{\rm int}(T))}^2 }\quad\text{for}\quad  \vv_h\in \CR^p_0(\T;\R^2),
\end{equation}
which is a norm %in $\CR^p_0(\T;\R^2)$
owing to the discrete Friedrichs inequality \cite[p. 299]{scottbrenner3}; $\text{\rm int}(T)$ is the interior of a (compact) triangle $T$  (the convex hull of its three vertices).
The  reciprocal $  \beta_p=1/C_p>0$  of the operator norm of the right-inverse of \eqref{eqccintro1}
is known as an $\inf$-$\sup$ constant
\[
 \beta_p= \inf_{q_h\in P_{p-1}(\T)\setminus \R} \sup_{\vv_h\in \CR^p_0(\T;\R^2) } \frac{ \int_\Omega q_h\ddiv_{\pw}\vv_h dx }{ \| q_h \|_{L^2(\Omega)} \trinorm{\vv_h}_\pw}.
\]
The point is that a positive lower bound of the $\inf$-$\sup$ constant $\beta_p$ is independent of the mesh-size in the triangulation.
Two universal positive constants $\varepsilon$ and $M$ describe the
dependency of the  positive lower bound of $ \beta_p$ on the triangulation $\T$ in
a set of admissible triangulations $\TT$.

\begin{definition}[admissible triangulations]\label{defadmissibletriangulations}
Given the positive constants $\varepsilon$ and $M$, let
$\TT$ be  the class of all regular triangulations $\T$ (of the domain $\Omega$ into triangles)
with at least one interior vertex and all interior angles in a triangle $K\in\T$ are bounded from
below by $\epsilon>0$. Furthermore, any triangle $K\in\T$ has either  at least one
vertex in the interior of the domain $\Omega$ (then $M:=0$), or there is a finite sequence of at most $ M\ge 1$
triangles  $K_0,K_1,\dots,K_k\in\T $ such that $k\le M$,
$\partial K_j\cap \partial K_{j+1}$ is a common edge of $K_j$ and $K_{j+1}$ for all $j=0,\dots, k-1$, $K=K_0$, and $K_k$ has at least one vertex in the interior of $\Omega$.
\end{definition}

Each triangulation $\T\in\TT$ in the uniformly shape-regular class of admissible triangulations $\TT$  has interior vertices and not too many triangles in $\T$
have all three vertices  on the boundary $\partial\Omega$ of the domain $\Omega$.

\begin{theorem}\label{thm1} For all $p\in\NN$ there exists a constant $\beta>0$
such that, for all $\T\in\TT$,
the piecewise divergence operator \eqref{eqccintro1} is surjective
and has a bounded inverse in the sense that
\[
\forall g\in \P_{p-1}(\T)/\R \; \exists \vv_h\in \CR^p_0(\T;\R^2)\; g=\ddiv_\pw \vv_h\text{ and } \beta_p\trinorm{\vv_h}_\pw\le \|g\|_{L^2(\Omega)}.
\]
(The constant $\beta_p$ exclusively depends on $p,\epsilon$, and $M$ in the definition of $\TT$.)
\end{theorem}

The discussion in the introduction about the existing references  \cite{Baran_Stoyan,CCSS_CR2,CrouzeixRaviart,Fortin_Soulie} shows
that  the remaining  critical case  $p=3$  is examined below.

   \begin{figure}
       \centering
       \input{nodal_patch}
       \caption{Nodal patch $\omega(\vz)$}
       \label{fig:nodal_patch}
   \end{figure}

\section{Elimination of interior degrees of freedom}\label{sectionEliminationofinteriordegreesoffreedom}
The analysis of conforming finite elements for the 2D Stokes problem simplifies
if one eliminates the contributions from the interior degrees of freedom \cite{GuzmanScott2019}.  The latter
are associated  to the cubic bubble-functions multiplied with vectors of polynomials of (total) degree at most $p-3$. We suppose
$p\ge 3$ in the sequel and abbreviate
the vector-valued bubble-functions, that vanish on the boundary $\partial T$ of the triangle $T$,  by
$\B_p(T):= P_p(T;\R^2)\cap H^1_0(\text{\rm int}(T);\R^2)$.

The remaining degrees of freedom are {four} (individually) for each triangle $T$  and any polynomial  $g\in P_{p-1}(T)$
of (total) degree at most $p-1$  as follows: The  evaluation $g(\vy)$ at the three vertices $\vy$ of  $T$ and
the integral $ \int_T g\, dx$.
Let $\La_T:P_{p-1}(T)\to\R^4$ denote the vector of those four values.

\begin{proposition}[Guzman-Scott]\label{propGuzman-Scott}
Any polynomial $g\in P_{p-1}(T)$ %in the triangle $T$ of degree at most $p-1$
is the divergence of a bubble-function in $\B_p(T)$  (i.e., $g=\ddiv \vb$ pointwise in $T$ for some $\vb\in \B_p(T)$) if and only if $\La_T(g)=\mathbf{0}$.
\end{proposition}

\begin{proof}
This is essentially known since \cite{MR696548}, allows a 3D version \cite{MR3356019}, and is given  with a combinatorial proof
in \cite[Prop. 2]{GuzmanScott2019}. The cubic case also follows from  elementary considerations with  $\ddiv \B_3(T)$.
\end{proof}

The proposition allows a  reduction to interpolating vertex values  \cite{GuzmanScott2019} in this paper as follows:
First we consider a patch $\omega(\vz) :=\textrm{int} (\cup\T(\vz)) $, covered by the $m$ neighbouring triangles  $\T(\vz):=\{ T\in \T: \vz\in\V(T)\}$
of an interior vertex  $\vz\in \V(\Omega)$
as depicted in  Figure~\ref{fig:nodal_patch}. %in the shape-regular triangulation $\T$.
(Here and throughout the paper $\V(T)$ denotes the set of the three vertices of the triangle $T$ and $\V(\Omega)$ denotes the set of all interior vertices in
the triangulation $\T\in\TT$.)
Second we investigate \eqref{eqccintro1} by the selection of an edge-oriented
basis of a linear subspace $\vV(\vz)$ of $\CR^p_0(\T(\vz);\R^2)$ defined below in \eqref{eqcccubicsection1} of dimension $5m$;
the precise definition of  $\vV(\vz)$  is not important in this section as long as it only intersects trivially  with the bubble-functions $ \B_p(\T(\vz))$ below.
Recall that $\La_T$ denotes the above four degrees of freedom
(the evaluation at the three vertices of plus the integral over the  triangle)
for each of the $m$ triangles   $T\in \T(\vz)$ and let  $\Lambda: P_{p-1}(\T(\vz))\to\R^{4m}$ denote
their patchwise version: Identify $\R^{4m}\equiv (\R^4)^m$ and  let $(\Lambda g)|_T:=\La_T(g|_T)\in\R^4$
for all  $T\in \T(\vz)$ and $g\in  P_{p-1}(\T(\vz))$.
The homogeneous boundary conditions (in the sense of Crouzeix-Raviart at the Gauss points of the edges along the boundary $\partial\omega(\vz)$)
of the discrete velocities
$\vvh \in \vV(\vz)\subset \CR^p_0(\T(\vz);\R^2)$ lead to
$\int_{\omega(\vz)}\ddiv_\pw \vvh dx=0$, i.e.,  ${\rm  div}_\pw: \vV(\vz)\to P_{p-1}(\T(\vz))/\R$.
Hence the composition
$\Lambda\circ {\rm  div}_\pw: \vV(\vz)\to \R^{4m}$ is a linear operator and its range
is  a subspace  of $\R^{4m} $ orthogonal to %perpendicular to the vector
\begin{equation}\label{eqccdefofs} %\eqref{eqccdefofs}-\eqref{eqccorthogonalityofs}
\vs:=(0,0,0,1,0,0,0,1,\dots, 1)=\sum_{j=1}^m \ve_{4j}\in\R^{4m}
\end{equation}
in terms of the $j$-th canonical unit vector  $\ve_j\in \R^{4m} $.
%$:= \{  x\in \R^{4m}: x(4)+x(8)+\cdots x(4m)=0\} $,
%when $x(j)$ denotes the $j$-th component  (in global enumeration)
In fact, given  $\vvh \in \vV(\vz)$ with $\vx:=\Lambda\, {\rm  div}_\pw \vvh \in \R^{4m}$,
the component  $x_{4j}=\int_{T(j)} {\rm  div}\, \vvh dx$
is the integral of $ {\rm  div}\, \vvh $
over the triangle $T(j)$ of number $j=1,\dots,m$. Then % sum of all of those components
\begin{equation}\label{eqccorthogonalityofs}
\vx\cdot \vs=\sum_{j=1}^m\int_{T(j)} {\rm  div}\, \vvh dx=\int_{\omega(\vz)}\ddiv_\pw \vvh dx=0
\end{equation}
is the vanishing integral over the patch $\omega(\vz)$. %:=\text{\rm int}(\cup\T(\vz))$.
Let  $ \B_p(\T(\vz))$ $\equiv \bigoplus_{j=1}^m  \B_p(T(j))\subset \CR^p_0(\T;\R^2)$ denote the piecewise bubble-functions  in the patch $\T(\vz)$.
This and Proposition~\ref{propGuzman-Scott} clarify that the piecewise divergence operator
\begin{equation}\label{eqccpiecewisedivergenceoperatorlocal}
{\rm  div}_\pw: \vV(\vz)\oplus \B_p(\T(\vz))\to P_{p-1}(\T(\vz))/\R
\end{equation}
is a linear map {\em into}  the pressure space $P_{p-1}(\T(\vz))/\R $ with zero integral.

\begin{lemma}\label{lemma:ccsas1}
Suppose  $ \vV(\vz)$ is some %($5m$-dimensional)
linear subspace of  $\CR^p_0(\T(\vz);\R^2)$
such that  the linear map  $\Lambda\circ {\rm  div}_\pw: \vV(\vz)\to \vs^\perp$ is surjective
{\em onto} the $(4m-1)$-dimensional orthogonal complement $\vs^\perp$ of the vector  $\vs$ in $\R^{4m}$. Then
the piecewise divergence operator \eqref{eqccpiecewisedivergenceoperatorlocal}
is  surjective {\em onto}  $P_{p-1}(\T(\vz))/\R $.
\end{lemma}

\begin{proof}
Given any $f\in P_{p-1}(\T(\vz))/\R $ with $ \vx:= \Lambda f\in\R^{4m}$,
$\int_{\omega(\vz)} f\, dx=0$ implies $\vx\cdot \vs=0$; written $\vx\in \vs^\perp$.
The surjectivity of  $\Lambda\circ {\rm  div}_\pw: \vV(\vz)\to \vs^\perp$ leads to some
 $\vvh \in \vV(\vz)$ with $\vx=\Lambda\, {\rm  div}_\pw \vvh $. Consequently,
 $g:=f- {\rm  div}_\pw \vvh\in  P_{p-1}(\T(\vz))/\R $ satisfies $\Lambda g=\mathbf{0}$.
Proposition~\ref{propGuzman-Scott} applies to any $g|_T$ for $T\in\T(\vz)$
and leads to some $\vb\in\B_p(\T(\vz))$ with ${\rm  div}_\pw\vb=g$
 in $\omega(\vz)$. Hence $f={\rm  div}_\pw( \vvh+\vb)$.%  in $\omega(\vz)$.
\end{proof}

\section{Cubic polynomials and geometry of an interior patch}
\label{sectionInteriorpatch}
Figure~\ref{fig:nodal_patch} displays the closed nodal patch of an interior  vertex $\vz$ in the shape-regular triangulation
$\T$, sometimes also called star,
that consists of $m\ge 3$ triangles $\T(\vz)=\{T(1),\dots,T(m)\}$ with
interior vertex $\vz$ and the vertices $\vP(1),\dots, \vP(m)$ on the boundary
$\partial\omega(\vz)$; $\T(\vz)$ is a regular triangulation of  the open and simply connected domain
$\omega(\vz)\equiv \text{\rm int}(\cup\T(\vz))$ covered by the neighbouring
triangles  $T(1),\dots,T(m)$ around $\vz$. A typical triangle $T(j)$ has the vertices $\vz$, $\vP(j)$,
and $\vP(j+1)$ counted counterclockwise as in Figure~\ref{fig:nodal_patch} such that
$\partial T(j-1)\cap \partial T(j)=E(j)=\text{\rm conv}\{\vz,\vP(j)\}$ is a common
edge of the consecutive triangles $T(j-1)$ and $T(j)$ for all $j=1,\dots, m$;
here and throughout a cyclic notation  is understood with $T(0)\equiv T(m)$,
$\vP(0)\equiv \vP({m})$, etc.
The vertex-oriented $P_{1}$ nodal basis functions $\varphi_{z}\in P_{1}\left(
\mathcal{T}\left(  \mathbf{z}\right)  \right)  \cap H_{0}^{1}\left(
\omega\left(  \mathbf{z}\right)  \right)  $ and $\varphi_{1},\ldots
,\varphi_{m}\in P_{1}\left(  \mathcal{T}\left(  \mathbf{z}\right)  \right)
\cap H^{1}\left(  \omega\left(  \mathbf{z}\right)  \right)  $ are the nodal
basis functions defined by $\varphi_{j}\left(  \mathbf{P}_{k}\right)
=\delta_{jk}$ (with Kronecker's $\delta_{jk}$) and $\varphi_{j}\left(
\mathbf{z}\right)  =0=\varphi_{z}\left(  \mathbf{P}_{k}\right)  $ for all
$j,k=1,\ldots,m$, while $\varphi_{z}\left(  \mathbf{z}\right)  =1$.
The  Crouzeix-Raviart  function for  the edge $E(j)$
is defined with the cubic Legendre polynomial $\Le_3$ by
\begin{equation}\label{eqcccpsiinubicsection}
\psi_j(x):= \Le_3( 1-2\varphi_k(x))\quad\text{at }x\in T(j)\text{ with }k=j+1
\end{equation}
 resp.\  at $x\in T(j-1)$  with $k=j-1 $
and zero elsewhere.
Then the  subspace $\vV(\vz)$ is
defined as the span of all the following $5m$ vector-valued functions
\begin{equation}\label{eqcccubicsection1}
\varphi_j\varphi_z^2 \, \vn(j),\hfil
\varphi_j\varphi_z^2 \, \vt(j),\hfil
\varphi_j^2\varphi_z\, \vn(j),\hfil
\varphi_j^2\varphi_z \, \vt(j),\hfil
\psi_j\, \vn(j)\in  \CR^p_0(\T(\vz);\R^2),
\end{equation}
also enumerated as
$\vb(5j-4):=\varphi_j\varphi_z^2 \, \vn(j), \dots,  \vb(5j):=\psi_j\, \vn(j)$,
where
 $\vn(j)$ is the unit normal and  $\vt(j):=(\vz-\vP(j))/|E(j)|$  the tangential unit vector
of the edge $E(j)=\text{\rm conv}\{\vz,\vP(j)\}$ of  length  $|E(j)|=|\vz-\vP(j)|>0$
with number $j=1,\dots,m$ as depicted in Figure~\ref{fig:nodal_patch}.

Recall the
definition of $\La_T$ from  Section~\ref{sectionEliminationofinteriordegreesoffreedom} for any $T\in\T(\vz)$ and enumerate all those
$4m$ components also by $(\Lambda_{4j-3},\dots,\Lambda_{4j}) := \La_{T(j)}$ for
$j=1,\dots,m$.
The resulting Vandermonde matrix
\begin{equation}\label{eqcccubicsection2matrixM}
\bM= (\La_j(\ddiv_\pw \vb(k)))_{k=1,\dots,5m}^{j=1,\dots,4m}\in \R^{5m\times 4m}
\end{equation}
will be calculated explicitly in the next section.
The following duality argument allows an application of  Lemma~\ref{lemma:ccsas1}.
Recall  $\bM\vs=\mathbf{0}$  from \eqref{eqccdefofs}-\eqref{eqccorthogonalityofs}. %the discussion right before Lemma~\ref{lemma:ccsas1}.

\begin{lemma}\label{lemma:ccsas2}
If the kernel %$\ker\bM=\sspan\{ \vs\}$
of the matrix $\bM$ is the span of $\vs$, then
 the linear map  $\Lambda\circ {\rm  div}_\pw: \vV(\vz)\to \vs^\perp$ is surjective
onto  $\vs^\perp\subset \R^{4m}$.
\end{lemma}

\begin{proof}
Notice that $\bM$ is the already the transpose of the coefficient matrix
$\bM^t$ $=(\La(\ddiv_\pw \vb(k)):k=1,\dots,5m)$
associated to the operator  $\Lambda\circ {\rm  div}_\pw: \text{\rm span}\{  \vb(1),\dots, \vb(5m)\}
\to \vs^\perp$ and the images $\im(\Lambda\circ {\rm  div}_\pw)=\im \bM^t$ coincide.
The known orthogonal decomposition  $\ker \bM\operp \im \bM^t=\R^{4m}$
into the kernel  $\ker \bM$ of the matrix $\bM$ and the image
 $\im \bM^t$ of its transpose $\bM^t$
is  the fundamental  theorem of linear algebra \cite{GStrang} included in textbooks  on linear algebra, e.g., in the form
$\ker\bM=(\im\bM^t)^\perp$ in \cite[Eq. (2.47)]{Greub}.
With $\ker\bM=\sspan\{ \vs\}$ this implies  the assertion $\im(\Lambda\circ {\rm  div}_\pw)=\im \bM^t=\vs^\perp$.
\end{proof}

This section  concludes with  general observations on the geometry
of an interior patch. The inner angles of the triangle  $T(j)$ at the respective vertices $\vz$, $\vP(j)$,
and $\vP({j+1})$ are denoted by $\omega(j)$, $\alpha(j)$, and $\beta(j)$.
The area $|T(j)|$ of the triangle $T(j)$
and the length $|E(j)|$ of an edge $E(j)$ displayed in Figure~\ref{fig:nodal_patch} are related to certain cotangent sums;
namely, for $j=1,\dots,m$,
\begin{align*}
\gamma_j^-&:=\frac{  |E(j)|^2 }{2\,|T(j-1)|}=\cot\omega(j-1)+\cot\beta(j-1)>0,
\\
\gamma_j^+&:=\frac{  |E(j)|^2 }{2\,|T(j)|}=\cot\omega(j)+\cot\alpha(j)>0.
\end{align*}
(The cyclic notation applies for
$\omega(0)\equiv \omega(m)$ and $\beta(0)\equiv\beta(m)$;
the above geometric identities in terms of angles in a triangle follow
from elementary geometry in a triangle.)
The combination of the above
leads to
\begin{equation}\label{eqngeometry1}
\gamma_j:=\gamma_j^-+\gamma_j^+=\kappa_j+\mu_j>0
\end{equation}%
%Since  $ \omega(j-1)+\beta({j-1})$ and $\omega(j)+ \alpha({j}) <\pi$ for any $j=1,\dots,m$,
%the subsequent  terms are positive
%\[
% \cot \omega(j-1)+\cot \beta({j-1}) >0
% \quad\text{and}\quad
% \cot \omega(j)+\cot \alpha({j}) >0 Adven
%\]
with real coefficients
\[
\kappa_j:=\cot \alpha(j)+\cot \beta({j-1})
\quad\text{and}\quad
\mu_j:=\cot \omega(j-1)+\cot \omega({j}) .
\]
The cotangent sums $\kappa_j$ and $\mu_j$ may become zero or negative;  but
$\kappa_j+\mu_j>0$ is positive{\color{blue}  :} the two quantities cannot vanish simultaneously.

A final observation concerns the  situation $\kappa_j=0$ that may happen
repeatedly. But the convex hull of the compact patch $\overline{\omega(\vz)}$
contains at least three extreme points and those are
convex corners of $\omega(\vz)$.  Hence there exist at least three vertices $\vP(j)$ with $\kappa_j>0$.

\section{The coefficient matrix}
\label{sectionThecoefficientmatrix}
This section derives an explicit closed formula for the Vandermonde matrix $\bM$ from \eqref{eqcccubicsection2matrixM} in
terms of a cyclic block-bi-diagonal matrix
\begin{equation}\label{eqcccubicsection4matrixM}
\bM=\begin{pmatrix} \bM_1^+ &&& \bM_1^- \\
 \bM_2^- &  \bM_2^+ \\
 &\ddots&\ddots \\
 && \bM_m^- & \bM_m^+
\end{pmatrix}\in \R^{5m\times 4m}.
\end{equation}
Empty positions in \eqref{eqcccubicsection4matrixM} are filled by zero and  $ \bM_j^\pm$ is a  $5\times 4$ matrix computed in the subsequent lemma,
where the  $5\times 8$ matrix  $( \bM_j^- , \bM_j^+)$ denotes the composition of $ \bM_j^-$ and $ \bM_j^+$.

\begin{lemma}\label{lemma:ccsas3}
The Vandermonde matrix $\bM$ %from  \eqref{eqcccubicsection2matrixM}
is the  block-diagonal matrix \eqref{eqcccubicsection4matrixM}
with the typical $5\times 8$ block $( \bM_j^- , \bM_j^+)$ for $j=1,\dots,m$.
The diagonal scaling
\begin{align*}
\bD_L(j)&:=\text{\rm diag}(|E(j)|,-|E(j)|,|E(j)|,|E(j)|,|E(j)|/12) \in\R^{5\times 5}, \\
 \bD_R(j)&:=
\text{\rm diag}(1,1,1, 6/|T(j-1)|,1,1,1, 6/|T(j)|)\in\R^{8\times 8},
\end{align*}
leads to the $5\times 8$
matrix $\bD_L(j)( \bM_j^- , \bM_j^+) \bD_R(j)$ that is  equal to
\[
\begin{pmatrix}
 \cot\omega(j-1) \hspace{-1mm} & 0 & 0 &\gamma_j^-
         & -\cot\omega(j)  \hspace{-1mm} & 0 & 0 & - \gamma_j^+  \\
1 & 0 & 0 & 0 & 1 & 0 & 0 & 0 \\
0 & 0 & \hspace{-1mm} \cot\beta(j-1) \hspace{-1mm}  &\gamma_j^- &0
   & \hspace{-1mm} -\cot\alpha(j)  \hspace{-1mm}& 0 & - \gamma_j^+ \\
0 & 0 & 1  & 0 & 0 & 1 & 0 & 0 \\
 \gamma_j^- & \gamma_j^- & \gamma_j^- & \gamma_j^-  &
 - \gamma_j^+&- \gamma_j^+&- \gamma_j^+&- \gamma_j^+
\end{pmatrix}.
\]
\end{lemma}

(Recall the cyclic convention  $(\omega(0),\beta(0))\equiv
(\omega(m),\beta(m))$ for $j=1$.)

\begin{proof}
This outline of the  proof guides the reader through the  arguments from elementary geometry in triangles. %
The functionals in $\La$ from
Section~\ref{sectionEliminationofinteriordegreesoffreedom}
consist of point evaluation or integrals of the piecewise divergence
of the ansatz  functions
$\vb(5j-4)\equiv\varphi_j\varphi_z^2 \, \vn(j)$, \dots,  $\vb(5j)\equiv\psi_j\, \vn(j)$.
The block structure of $\bM$
results from the support $T(j-1)\cup T(j)$ of those $5$ piecewise polynomials.

The first component of the functional $\La$ associated to the triangle $T(j)$ is the point evaluation at the vertex $\vz$  of a function
$\ddiv(\varphi_j\varphi_z^2 \,\vv )=\vv\cdot\nabla(\varphi_j\varphi_z^2)$
resp.\  $\ddiv(\varphi_j^2\varphi_z \,\vv )$
for some constant vector $\vv\in\R^2$ and this is equal to  $\vv\cdot\nabla\varphi_j|_{T(j)}$
resp.\  zero.
The (global) nodal basis functions $\varphi_z$ and
$\varphi_j$ are barycentric coordinates in each triangle and their derivatives are known from the geometric
fundamentals of the (conforming) $P_1$ FEM:
 $|E(j)|\, \vn(j)\cdot \nabla\varphi_z$ and  $|E(j)|\, \vn(j)\cdot \nabla\varphi_j$ are equal to
$-\cot\alpha(j)$ and $-\cot \omega(j)$ in $T(j)$  (resp.\  $\cot\beta(j-1)$ and $\cot \omega(j-1)$ in $T(j-1)$).
The tangential components $|E(j)|\, \vt(j)\cdot \nabla\varphi_z=1$
and $|E(j)|\, \vt(j)\cdot \nabla\varphi_j=-1$ are constant in $T(j-1)\cup T(j)$.
 This leads to the entries $M_j^-(1,1)=|E(j)|^{-1}\, \cot\omega(j-1)$ and  $M_j^+(1,1)=-|E(j)|^{-1}\, \cot\omega(j)$
in the first position of $M_j^\pm\in \R^{5\times 4}$ and to $M_j^-(2,1)=-|E(j)|^{-1}=M_j^+(2,1)$.

Similar arguments apply to the point evaluations at the
other vertices $\vP(j)$ and $\vP({j+1})$ and we obtain, e.g.,
$M_j^-(3,3)=|E(j)|^{-1}\, \cot\beta(j-1)$, $M_j^+(3,2)=-|E(j)|^{-1}\, \cot\alpha(j)$, and  $M_j^-(4,3)=|E(j)|^{-1}=M_j^+(4,2)$.

The situation is slightly  different for the edge-bubble
function $\psi_j$ from \eqref{eqcccpsiinubicsection}: The one-dimensional
derivative $\Le'_3(\pm1)=6$   of the cubic Legendre polynomial $\Le_3(t)=(5t^3-3t)/2$
is the same at the end-points of the reference interval $[-1,1]$. Thus  the
point evaluations of  $\ddiv (\psi_j(x)\,\vn(j))=\vn(j)\cdot \nabla\psi_j(x)=
-2\Le'_3(1-2\varphi_k(x)) \vn(j)\cdot  \nabla\varphi_k$ at the three vertices of $T(j-1)$ (resp.\  $T(j)$) coincide with
$-12 \,\vn(j)\cdot  \nabla\varphi_k$.
Since  $\varphi_{j+1}$ vanishes along $E(j)$, its gradient
$\nabla\varphi_{j+1}|_{T(j)}=\vn(j)/h_{j,j+1}$
is perpendicular to $\vt(j)$, i.e.,  parallel to $\vn(j)$, and has a length $h_{j,j+1}^{-1}$, where
 $h_{j,j+1}=2\,|T(j)|/|E(j)|$ is the height of the vertex $\vP(j+1)$  onto the edge $E(j)$
in the triangle $T(j)$.
Consequently   $\ddiv (\psi_j\,\vn(j))$ assumes
the value    $ 6\,|E(j)|/|T(j-1)|=12\,\gamma_j^-/|E(j)|$  at any vertex of $T(j-1)$
(resp.\   the value $ -6\,|E(j)|/|T(j)|=-12\,\gamma_j^+/|E(j)|$  at any vertex of $T(j)$).
Given any vector $\vv\in\R^2$, the integrals over $T(j)$ allow
an integration by parts
\begin{align*}
&\int_{T(j)} \ddiv(\varphi_j\varphi_z^2 \,\vv )dx
=\vv \cdot \int_{\partial T(j)} \varphi_j\varphi_z^2\vn_{T(j)} dx \\
&\quad = - \vv\cdot\vn(j)\,\int_{E(j)}  \varphi_j\varphi_z^2ds=-\vv\cdot\vn(j) {|E(j)|}/12
\end{align*}
(with the exterior unit normal $\vn_{T(j)} $ %along  $\partial T(j)$
and  $\int_0^1 s^2(1-s)\, ds=1/12$).  The same argument applies to
$-\int_{T(j)} \ddiv(\psi_j \,\vn(j))dx$ $=|E(j)|=\int_{T(j-1)} \ddiv(\psi_j \,\vn(j))dx$.

The summary of the aforementioned identities leads to the assertion.
\end{proof}

Lemma~\ref{lemma:ccsas3}
suggests a global scaling of the matrix $\bM$ with the block-diagonal matrices
\begin{align*}
 \bD_L&:=\text{\rm diag}(\bD_L(1),\dots,\bD_L(m))\in \R^{5m\times 5m} ,\\
\bD_R&:=\text{\rm diag}( \bD_R'(1),\dots,\bD_R'(m))\in \R^{4m\times 4m}
\end{align*}
with
$ \bD_L(j) \in\R^{5\times 5}$  as above and
$\bD_R'(j):=\text{\rm diag}(1,1,1, 6/|T(j)|)\in\R^{4\times 4}$.
The investigation of the kernel of $\bM$ is then rewritten in that of
the cyclic block-bi-diagonal matrix
\begin{equation}\label{eqcccubicsection4matrixA}
\bA:= \bD_L \bM\bD_R =\begin{pmatrix} \bA_1^+ &&&  \bA_1^- \\
 \bA_2^- &  \bA_2^+ \\
 &\ddots&\ddots \\
 && \bA_m^- & \bA_m^+
\end{pmatrix}\in \R^{5m\times 4m}
\end{equation}
with the blocks $ (\bA_j^-,\bA_j^+):=\bD_L(j)( \bM_j^- , \bM_j^+) \bD_R(j)$
displayed in Lemma~\ref{lemma:ccsas3}.

A brief observation
on the constant pressure values concludes
this section.
%Recall the above vector $\vs$ from Lemma~\ref{lemma:ccsas2} with  $\bM\vs=0$ belongs to the kernel of $\bM$ and so
Recall  $\bM\vs=\mathbf{0}$  from \eqref{eqccdefofs}-\eqref{eqccorthogonalityofs} so that
\[
\vv_0:=6\bD_R^{-1}\vs= (0,0,0,|T(1)|,0,0,0,|T(2)|,\dots,0,0,0 ,|T(m)|)^t\in \R^{4m}
\]
belongs to the kernel of $\bA$. If the dimension of the kernel of $\A$
is smaller than or equal to one, then it is one (because $\A\vv_0=\mathbf{0}$) and
the kernel  $\bM$ is  $\text{\rm span}\{\vs\}$ and
Lemma~\ref{lemma:ccsas2} applies. The proof that the kernel of $\A$ has
dimension at most one is the technical key and the content of Section~\ref{sectionThekernelofbAisone-dimensional}.

\section{The kernel of $\bA$ is one-dimensional}
\label{sectionThekernelofbAisone-dimensional}
The  characterization of the kernel of $\A$ follows in three
steps. The {\em first step} reduces the problem to a sub-matrix and concerns  the row number two, four, and five of
$\bA_j^-$ and $\bA_j^+$  for each $j=1,\dots,m$; i.e.,
\begin{equation}\label{eqcccubicsection5matrixB}
\bB:=\bA(2,4,5,7,9,10,\dots,5m;1 ,\dots, 4m)\in\R^{3m\times 4m}
\end{equation}
with the $3m$ rows of number $  k=5j-3,5j-1,5j$ for $j=1,\dots,m$. (Standard notation for sub-matrices from matrix analysis applies
throughout this paper: $A(I;J)$ denotes the (rectangular) sub-matrix of $A$ with the entries $A(i,j)$ for $(i,j)\in I\times J$ specified in the index lists $I$ and $J$.)
Let  $\ve_j$ be the $j$-th canonical  unit vector in $\R^{4m}$.

Define the  %(linearly independent)
vectors $\vv_0,\dots,\vv_{m+1}\in \R^{4m}$  by
\begin{align*}
\vv_0&\equiv  \sum_{j=1}^{m} |T(j)|\, \ve_{4j}, \quad
\vv_{1}:=-\ve_{2}+\ve_{4}+\ve_{4m-1}-\ve_{4m} ,\\
\vv_j &:= \ve_{4j-5}-\ve_{4j-4}-\ve_{4j-2}+\ve_{4j} \text{ for any  }
j=2,\dots, m,\\
\vv_{m+1}&:=\sum_{j=1}^m (-1)^{j+1}( \ve_{4j-3}-\ve_{4j}).
\end{align*}

\begin{lemma}\label{lemma:ccsas4}
 The kernel of $\bB$ from  \eqref{eqcccubicsection5matrixB} has the dimension  $m+1+\sigma $
with $\sigma=0$ if $m=3,5,7,\dots $ is odd and $\sigma=1$ if $m=4,6,8,\dots$ is even.
The vectors $\vv_0,\dots,\vv_{m+\sigma}\in \R^{4m}$ form
a basis of the kernel $\ker\B$ of $\bB$.
\end{lemma}

The proof of  Lemma~\ref{lemma:ccsas4} by straightforward linear algebra  will be postponed to
Section~\ref{secProofofLemmareflemma:ccsas4}
in order not to delay the flow of arguments in the proof.

\medskip

{\em Step two} of the proof recalls
that $\vv_0$ is a known kernel vector of $\A$ and so can and will  be neglected in the sequel.
The aim is a  characterization of all remaining kernel vectors
$\vx$ of $\A$ as a  linear combination
\begin{equation}\label{eqncruciala1}
\vx=\eta\,\vv_{m+1}+\sum_{j=1}^{m} \xi_j \vv_j \in \R^{4m}
\end{equation}
of the other kernel vectors $\vv_1,\dots,\vv_{m+\sigma}$
of $\B$ with real coefficients  $\xi_1,\dots,\xi_m$ and $\eta$ and the following convention:
$\eta:=0$ if $\sigma=0$ and $m=3,5,7,\dots $ is odd, while
$\eta\in\R$ if $\sigma=1$ and $m=4,6,8,\dots $ is even.
Suppose throughout  this section that  $\vx$ from \eqref{eqncruciala1} satisfies $\bA\vx=\mathbf{0}$.
The  claim of this section is that  $\vx=\mathbf{0}$.
Recall  $\bB\vx=\mathbf{0}$  by  Lemma~\ref{lemma:ccsas4} and consider the remaining equations
 $0=(\bA\vx)( 5j-4)$ and   $0=(\bA\vx)( 5j-2)$  of
$\A\vx\in \R^{5m}$ in the position  $ 5j-4 $ and $5j-2$
for each $j=1,\dots,m$, which are equivalent to
 \begin{align}\label{cceqn51a}
(-1)^{j} \kappa_j \eta &=  \gamma_j^-\xi_{j-1} -\gamma_j \,\xi_{j}
+ \gamma_j^+\xi_{j+1} ,\\
(-1)^{j} \gamma_j \eta&= \label{cceqn51b}
\gamma_j^-\xi_{j-1} -\mu_j\, \xi_{j}+ \gamma_j^+\xi_{j+1}
\end{align}
with \eqref{eqngeometry1} and the geometric terms from the end of Section~\ref{sectionInteriorpatch}.
(Follow the cyclic convention $\xi_0\equiv\xi_m$ and $\xi_{m+1}\equiv\xi_1$ where necessary.)
Define the convex coefficients
$0<\lambda_j:= \gamma_j^-/\gamma_j<1$ and
$0<1-\lambda_j= \gamma_j^+/\gamma_j<1$ to rewrite \eqref{cceqn51a}-\eqref{cceqn51b},
for all $j=1,\dots,m$, as
\begin{align}\label{eqncrucial1}
 \lambda_j\xi_{j-1} - \xi_{j}+  (1-\lambda_j)\xi_{j+1}
&= (-1)^{j}( \kappa_j/\gamma_j )\,  \eta ,\\   \label{eqncrucial2}
\lambda_j \xi_{j-1} -(\mu_j/\gamma_j) \xi_{j}+ (1-\lambda_j)\xi_{j+1}&=
(-1)^{j}  \eta.
\end{align}
The aim is to verify that those $2m$ equations
\eqref{eqncrucial1}-\eqref{eqncrucial2}
cannot hold simultaneously unless
all the real coefficients  $\xi_1,\dots,\xi_{m}$ and $\eta$ vanish. The indirect proof
assumes that $J:=\{ j=1,\dots,m: \xi_j\ne 0\}$ is non-empty and then leads
to a  contradiction. (Then $J=\emptyset$ and \eqref{eqncrucial2}
imply $\eta=0$ as well, whence $\vx=\mathbf{0}$.)
Let us fix the notation $J^C:=\{1,\dots,m\}\setminus J$ before we
distinguish two cases  $\eta=0$ and $\eta\ne 0$ in {\em step three}  of the proof.

\bigskip

{\em First case  $\sigma=0$ or  $\sigma=1$ with $\eta=0$.}
Then the right-hand sides in \eqref{eqncrucial1}-\eqref{eqncrucial2} vanish and the difference of the two equations shows
\[
0=(1-\mu_j/\gamma_j) \xi_{j}=(\kappa_j/\gamma_j) \xi_{j}
\quad\text{for all }j=1,\dots, m.
\]
This provides no new information for $j\in J^C$, but  it implies
\[
\kappa_j=0\quad\text{for all }j\in J.
\]
The geometry of the nodal patch of $\vz$ enters at this point as an additional argument.
Section~\ref{sectionInteriorpatch} closes with the existence
of at least three vertices $\vP(j)$ with $\kappa_j>0$,
whence $J^C$ contains at least three indices. The hypothesis $J\ne\emptyset$
leads to $m\ge 4$ and at least one index in $J$. Hence we can find $k$ consecutive indices
$j+1,\dots, j+k\in  J$  with  $j, j+k+1\in J^C$ --- interpreted
in a cyclic notation (with index $0$ identified with $m$ etc.).
Changing the enumeration
of the vertices (if necessary) we find $k\in \{1, \dots, m-3\}$ such that (in the new
enumeration) $\xi_m=0=\xi_{k+1} $ while  $\xi_1,\dots,\xi_k \in\R\setminus\{0\}$ do not vanish.
The homogeneous equation  \eqref{eqncrucial1} holds for $j=1,\dots, k$ and this can be written simultaneously
as a linear system of $k$ equations
\begin{equation}\label{eqncrucial4}
\begin{pmatrix}
1 & \lambda_1-1 \\
-\lambda_2 & 1 & \lambda_2-1 \\
& \ddots & \ddots &\ddots \\
&& -\lambda_{k-1} & 1 & \lambda_{k-1}-1 \\
&&&-\lambda_k & 1
\end{pmatrix}
\begin{pmatrix}
\xi_1\\
\xi_2\\
\vdots\\
\xi_{k-1}\\
\xi_k
\end{pmatrix}
=\mathbf{0}%\in\R^k
\end{equation}
with a (possibly non-symmetric)  $k\times k$ tri-diagonal coefficient matrix. For $k=1$ and $k=2$,  \eqref{eqncrucial4}   exclusively allows the trivial solution
$(\xi_1,\dots,\xi_k)=\mathbf{0}$. For $k\ge 3$, the tri-diagonal matrix in  \eqref{eqncrucial4}
is irreducible and  weakly diagonally dominant,  but strictly diagonally dominant  in the first and last
row. Those matrices are  called  weakly chained diagonally dominant
and known to be regular; cf., e.g., \cite[Cor. 6.2.27]{hornjohnson} or  O.~Taussky's paper \cite{MR32557}.
It follows that  \eqref{eqncrucial4}   merely allows the trivial  solution
$(\xi_1,\dots,\xi_k)=\mathbf{0}$.  This contradicts  the assumption $1,\dots,k\in J$.
Hence $J=\emptyset$ and $\text{\rm ker}\,\bA=  \text{\rm span}\{ \vv_0 \}$. \qed % in the {\em first case}. \qed

\bigskip

{\em Second case  $\sigma=1$ and $\eta\ne 0$.}
%,
Define the coefficients $\eta_j:=(-1)^{j} \xi_j/ \eta\in\R $ and
%notice $\eta_j=0$ if and only if $j\in J^C$.
substitute
$\xi_j:=  (-1)^{j}\eta\,\eta_j$ in
\eqref{eqncrucial1}-\eqref{eqncrucial2}
for all $j=1,\dots,m$. This leads, for all $j=1,\dots, m$,
 to the two identities
\begin{align}\label{eqncrucial3}
-\lambda_j\eta_{j-1} -  (1-\lambda_j)\eta_{j+1}
=  \kappa_j/\gamma_j  +\eta_j = 1+(\mu_j/\gamma_j) \eta_{j}.
\end{align}
With \eqref{eqngeometry1} and the geometric notions from the end of Section~\ref{sectionInteriorpatch}, the second equality in
\eqref{eqncrucial3}  is equivalent to
\[
\gamma_j=\kappa_j(1+\eta_j)>0.
\]
Hence  $\eta_j\ne-1$ and $\kappa_j\ne 0 $
for all $j=1,\dots, m$. The first equality in
\eqref{eqncrucial3} can be rewritten as
$\lambda_j\eta_{j-1} + c_j\eta_j+(1-\lambda_j)\eta_{j+1}=d_j$
with
\[
c_j:=\begin{cases} 0 & \text{ if } \eta_j=0 ,\\
1+1/(\eta_j(1+\eta_j)) & \text{ if }\eta_j\ne 0\end{cases}
\quad\text{and}\quad
d_j:=\begin{cases} -1 & \text{ if } \eta_j=0 ,\\
0 & \text{ if }\eta_j\ne 0\end{cases}
\]
for any $j=1,\dots,m$. All those $m$  conditions simultaneously form
a system of linear equations
 \begin{equation}\label{eqncrucial2a}
 \bT(\eta_1,\dots,\eta_m)^t=\mathbf{d}:=(d_1,\dots,d_m)^t\in \R^m
 \end{equation}
 with the  right-hand side  $\mathbf{d}$ %on the right-hand side
 and the cyclic tri-diagonal coefficient matrix
 \[
 \bT:=
\begin{pmatrix}
c_1 & 1-\lambda_1&&&\lambda_1 \\
\lambda_2 & c_2 & 1-\lambda_2 \\
& \ddots & \ddots &\ddots \\
&& \lambda_{m-1} & c_{m-1} &1- \lambda_{m-1} \\
1-\lambda_m &&&\lambda_m & c_m
\end{pmatrix} .
%\begin{pmatrix}
%\eta_1\\
%\eta_2\\
%\dots\\
%\eta_{m-1}\\
%\eta_m
%\end{pmatrix}
%=0.
\]
The compact row-wise Gershgorin discs $\overline{B}(c_j,1)$ in the coefficient matrix
$\bT$
have the center $c_j$ and the radius one.
%For $j\in J^C$, $c_j=1$ and $0\in \partial{B}(c_j,1)$ it holds $\eta_j=0$.
If  $j\in J$  and  so $0\ne \eta_j\ne -1$, then  $|c_j|>1$
(use $c_j=-1+\eta_j/(1+\eta_j)+(1+\eta_j)/\eta_j<-1$ in case $-1<\eta_j<0$).
Hence the Gershgorin disc $\overline{B}(c_j,1)\not\ni 0$ does {\em not} include zero.
A {\em first}  conclusion is that $J^C=\emptyset $ implies that $\bT$ is regular
(zero does not belong to any of the Gershgorin discs and so is not an eigenvalue)
and the right-hand side  $\mathbf{d}=\mathbf{0}$ vanishes.  Hence \eqref{eqncrucial3}
allows only the trivial solution and all coefficients $\eta_1,\dots,\eta_m$ vanish.
This  leads to $J=\emptyset$.
A {\em second} conclusion concerns  the remaining cases:  Suppose,
for a contradiction,  that
$J\ne \emptyset\ne J^C$ and delete all the rows and columns in the coefficient matrix
 $\bT$ and in the right-hand side $\mathbf{d}$  in \eqref{eqncrucial2a}
 with an index from $J^C$ to obtain the sub-matrix  $ \bT':=\bT(J,J) $.
Since $\eta_j$ vanishes for all $j\in J^C$ and $d_j=0$ for all $j\in J$, the  linear system
\eqref{eqncrucial2a} reduces to the homogeneous linear system
$\bT' \vy'=\mathbf{0}$ with the reduced coefficient vector $\vy':=(\eta_j:j\in J)$.
The row-wise compact Gershgorin discs $\overline{B}(c_j,r_j)$
for the cyclic tri-diagonal matrix $\bT'$ for any $j\in J$ have the center $c_j$ with
$|c_j|>1$ as before; but the radius $0\le r_j\le 1$ may be smaller than one in case
a corresponding row in $\bT$ has been deleted to obtain
$\bT'$. It follows that zero does
{\em not} belong to any  Gershgorin disc $\overline{B}(c_j,r_j)$ in $\bT'$ for $j\in J$ and so
$\bT'$ is a regular matrix. Thus the  reduced homogeneous system
$\bT' \vy'=\mathbf{0}$ %of linear equations
allows only the trivial solution $\vy'=\mathbf{0}$; whence $J=\emptyset$.

In all cases,  $J=\emptyset$, but that means a contradiction in \eqref{eqncrucial3}. Hence the second case with $\eta\ne 0$
cannot arise.
%This  leads to $\vx=\mathbf{0}$ and  the contradiction
%Consequently,  $\vx=\mathbf{0}$ in \eqref{eqncruciala1} and  $\text{\rm ker}\,\bA=  \text{\rm span}\{ \vv_0 \}$. % in the second case.
 \qed

\section{Proof of  Lemma~\ref{lemma:ccsas4}}\label{secProofofLemmareflemma:ccsas4}
The submatrix $\B$ of $\A$ from  \eqref{eqcccubicsection5matrixB} is the cyclic block-bi-diagonal matrix
\begin{equation}\label{eqcccubicsection6matrixB}
\bB =\begin{pmatrix} \bB_1^+ &&&  \bB_1^- \\
 \bB_2^- &  \bB_2^+ \\
 &\ddots&\ddots \\
 && \bB_m^- & \bB_m^+
\end{pmatrix}\in \R^{3m\times 4m}
\end{equation}
with the $3\times 8$ blocks, for  $j=1,\dots,m$,
\begin{equation}\label{eqnewcctocalculate}
( \bB_j^- ,  \bB_j^+)=
\begin{pmatrix}
1 & 0 & 0 & 0 & 1 & 0 & 0 & 0 \\
0 & 0 & 1  & 0 & 0 & 1 & 0 & 0 \\
 \gamma_j^- & \gamma_j^- & \gamma_j^- & \gamma_j^-  &
 - \gamma_j^+&- \gamma_j^+&- \gamma_j^+&- \gamma_j^+
\end{pmatrix}.
\end{equation}
Recall $\vv_0,\dots,\vv_{m+\sigma}\in \R^{4m}$ from Lemma~\ref{lemma:ccsas4}
in terms of  the canonical unit vectors $\ve_1,\dots, \ve_{4m}$ in $\R^{4m} $.

\medskip

{\em Claim 1: $\vv_0,\dots,\vv_{m+\sigma}$ are kernel vectors of $\bB$.}
This follows from straightforward calculations
of the scalar product of the three rows of
$( \bB_j^- ,  \bB_j^+)$ from \eqref{eqnewcctocalculate} for $j=1,\dots, m$
with $\vv_k := \ve_{4k-5}-\ve_{4k-4}-\ve_{4k-2}+\ve_{4k} $ for
$k=2,\dots, m$. The first row of $( \bB_j^- ,  \bB_j^+)$ has two ones only at positions where  $\vv_k$ vanishes;
the second row matches for $j=k$ at two positions with $1-1=0$ (and no interaction otherwise). The third row  vanishes because the sum of all components of $\vv_k$
in the positions $4j-3, 4j-2, 4j-3, 4j$ vanish for each $j=1,\dots, m$.  Further details are not displayed here
with two exceptions. The first one %detail
is that
\(
|T(j-1)|\, \gamma_j^-=|T(j)|\, \gamma_j^+=|E(j)|^2/2
\)
and so $\B\vv_0=\mathbf{0}$,  a rewriting of \eqref{eqccdefofs}-\eqref{eqccorthogonalityofs}.
The second  detail is
$\bB\vv_{m+1}= ( 1-(-1)^{m}) \ve_1$ to
underline that  $\vv_{m+1}$ is a kernel vector of $\bB$
if and only if $m$ is even. \qed

\medskip

{\em Claim 2: $\vv_0,\dots,\vv_{m+1}$ are linearly independent.}
Suppose that $\xi_0,\dots,\xi_{m+1} $ are the real coefficients of the trivial
linear combination $\mathbf{0}=\sum_{j=0}^ {m+1} \xi_j\, \vv_j $. Three observations
complete the  proof that all coefficients  $\xi_0,\dots,\xi_{m+1} $ vanish.
The {\em first} one  is that the scalar product
$\vv_j\cdot (1,1,1,1,0,\dots,0)=0$ vanishes
for any $j=1,\dots,m+1$. This and  $\vv_0\cdot (1,1,1,1,0,\dots,0)=|T(1)|$ show
$0=(\sum_{j=0}^ {m+1} \xi_j\, \vv_j )\cdot (1,1,1,1,0,\dots,0)=\xi_0|T(1)|$;
whence $\xi_0=0$.

The {\em second} observation is the scalar product
 $\vv_{j}\cdot \ve_{1}=0$ vanishes for all $j=1,\dots,m$, while
 $\vv_{m+1}\cdot \ve_{1}=1$.
 Thus $\xi_{m+1} =(\sum_{j=1}^ {m+1} \xi_j\, \vv_j)\cdot \ve_1=0$.

The {\em third} observation is that the scalar
product of  $\vv_j$ with   $\ve_{4k-2}$ reads
\begin{equation}\label{eqcclemmaproofa1}
\vv_j\cdot \ve_{4k-2}=-\delta_{jk} \quad \text{for all } j,k=1,\dots, m .
\end{equation}
%(with Kronecker's $\delta_{jk} $).
Hence  $0=(\sum_{j=1}^ {m} \xi_j\, \vv_j)\cdot \ve_{4k-2}=-\xi_k $
vanishes for all $k=1,\dots,m$. \qed

\medskip

{\em Claim 3: The kernel of $\bB$ is included in the span of
$\vv_0,\dots,\vv_{m+\sigma}$.} Suppose that $\vx=(x_1,\dots, x_{4m})\in\R^{4m}$
is any kernel vector
of $\bB$ and define
\[
\vy:=\vx+\sum_{j=1}^m x_{4j-2}\vv_j\in \R^{4m}.
\]
Equation \eqref{eqcclemmaproofa1} and the
definition of $\vy=(y_1,\dots,y_{4m})$  prove $y_{4j-2}=0$
for all $j=1,\dots,m$.
Claim~1 guarantees that $\vy$ is also a kernel vector of  $\bB$. The components
number $2,5,8,\dots, 3m-1$ of $\mathbf{0}=\bB\vy\in\R^{3m}$
read $0=y_{4j-1}+y_{4j+2}$ for all $j=1,\dots,m$ (with $y_{4m+2}:=y_2$).
Since all $y_{4j+2}=0$ vanish by design of $\vy$,
this implies $y_{4j-1}=0$ for  $j=1,\dots,m$.
The components
number $1,4,7,\dots, 3m-2$ of $\mathbf{0}=\bB\vy\in\R^{3m}$
read $0=y_{4j-3}+y_{4j+1}$ for all $j=1,\dots,m$ (with $y_{4m+1}:=y_1$).
These $m$ equations are collected in an equivalent matrix form as
\begin{equation}\label{eqcckeyproofb1}
\begin{pmatrix}
1 & &&1 \\
1 & 1  \\
& \ddots & \ddots \\ % \ ddots
&&1 & 1
\end{pmatrix}
\begin{pmatrix}
y_1\\
y_5\\
\vdots\\
y_{4m-3}\\
\end{pmatrix}
=\mathbf{0}\in\R^{m}
\end{equation}
with an $m\times m$ coefficient matrix $\mathbf{1}+\bF$ for the unit matrix $\mathbf{1}$
and a  Frobenius companion matrix $\bF$. The determinant
of $\mathbf{1}+\bF$  is $1-(-1)^m$  (this is well known and elementary to verify)
and this motivates the following two cases.
If $m\ge 3$ is odd and $\sigma:=0$, then the linear system \eqref{eqcckeyproofb1}
has only the trivial solution, i.e.,
$0=y_{4j-3}$  for $j=1,\dots, m$ and we set $\vz:=\vy\in \R^{4m}$.
If $m\ge 4$ is even and $\sigma=1$, then  $\mathbf{1}+\bF$ is singular with
a one-dimensional kernel spanned by the vector $(1,-1,+1,\dots,-1)\in\R^m$.
Consequently $y_{4j-3} = (-1)^{j+1} y_{1}$ for all $j=1,\dots,m$ and we define
\[
\vw:=\vy-y_{1}\vv_{m+1}\in \R^{4m}.
\]
A direct calculation shows $0=w_{4j-3}$ for all $j=1,\dots,m$.
In summary of the even or odd cases, we have  for any $m\ge 3$
that  the kernel vector $\vw=(w_1,\dots, w_{4m})\in\R^{4m}$ satisfies
$\vx-\vw \in  \text{\rm span}\{\vv_1,\dots,\vv_{m+\sigma}\}$ and
$0=w_{4j-3}=w_{4j-2}=w_{4j-1}$ for all $j=1,\dots,m$ (the vector $\vv_{m+1}$
does not interfere with the components number $2,3,6,7,10,11,\dots,4m-1$).
Thus %In other words,
 \[
  \vw =\sum_{j=1}^m  w_{4j}\ve_{4j}\in\R^{4m}
 \]
has non-zero entries at most in the components $4,8,\dots, 4m$.  The components
number $3,6,9,\dots, 3m$ of $\mathbf{0}=\bB \vw \in\R^{3m}$
read $0=\gamma_j^-w_{4(j-1)}-\gamma_j^+w_{4j}$ for all
$j=1,\dots,m$ (with $w_{0}:=w_{4m}$). The resulting $m$ relations
imply  $ w_{4j}= w_{4m} \prod_{k=j+1}^m  (\gamma_k^+ /\gamma_k^-)
= w_{4m} |T(j)|/|T(m)|$ for all $j=1,\dots,m$. Hence
\[
 \vw =(w_{4m}/|T(m)|)\, \vv_0.
\]
Claim~3 follows from
 $\vw , \vx- \vw \in \text{\rm span}\{\vv_0,\dots,\vv_{m+\sigma}\}$.
\qed

\section{Finish of the proof of Theorem~\ref{thm1}}\label{sectionProofofTheoremrefthm1}
\subsection*{Macro-element methodology}
The arguments  from   \cite{zbMATH03825382,Crouzeix_Falk,Stenberg_macro} are known to lead to Theorem~\ref{thm1} and are merely summarized.
The calculations in the previous sections characterize the kernel of the matrix $\bA$ from \eqref{eqcccubicsection4matrixA}
and then that of the Vandermode matrix $\bM$ from \eqref{eqcccubicsection2matrixM} and \eqref{eqcccubicsection4matrixM} as
the span $\ker\bM= \text{\rm span}\{\vs\}$ of the vector $\vs$ from \eqref{eqccdefofs}.
This and first Lemma~\ref{lemma:ccsas2} and second Lemma~\ref{lemma:ccsas1} prove that the piecewise
divergence operator \eqref{eqccpiecewisedivergenceoperatorlocal} is  surjective {\em onto}  $P_{p-1}(\T(\vz))/\R $.
This allows us to argue as in
the proofs of  \cite[Prop. 3.1, 5.1]{Crouzeix_Falk}  to conclude
that there exists a bounded right-inverse in the sense that
\begin{align}\nonumber
\forall g\in &\P_{p-1}(\T(\vz))/\R  \quad \exists \vv_h\in \CR^p_0(\T(\vz);\R^2)\quad g=\ddiv_\pw \vv_h\text{ and } \\
\label{eqfinalsection1asH1}
&\trinorm{\vv_h}_{ \omega(\vz),\pw }
:=\textstyle \sqrt{\sum_{T\in\T(\vz)} | \vv_h|_{H^1(\text{\rm int}(T))}^2}
\le \|g\|_{L^2(\omega(\vz))}/\gamma.
\end{align}
The operator norm of the bounded inverse
of \eqref{eqccpiecewisedivergenceoperatorlocal} and  so $\gamma>0$ in  \eqref{eqfinalsection1asH1}
exclusively depend on $p$ and the shape regularity  (i.e., on the constant  $\epsilon$ in the definition of $\TT$)
but {\em not} on the mesh-size.
%for some constant $\gamma>0$ and the piecewise $H^1$ norm   in $ \omega(\vz)$ defined by
%\trinorm{\vv_h}_{ \omega(\vz),\pw }=\textstyle \sqrt{\sum_{T\in\T(\vz)} | \vv_h|_{H^1(\text{\rm int}(T))}^2} .
%It appears to be  generally accepted that a scaling argument with a contra-variant Piola
%transform  (cf., e.g., \cite[p.~79]{Monkbook} or \cite[p.~106]{ErnGuermondFEM1})
%shows that the operator norm of the bounded inverse of \eqref{eqccpiecewisedivergenceoperatorlocal} and  so $\gamma>0$ in
%$\eqref{eqfinalsection1asH1}
%exclusively depend on $p$ and the shape regularity  (i.e., on the constant  $\epsilon$ in the definition of $\TT$)
%but {\em not} on the mesh-size.
Then \eqref{eqfinalsection1asH1} is, in different notation, what
 is called hypothesis H0 in \cite{Crouzeix_Falk} for a macro-element $\omega(\vz)$.

Recall from the definition of $\TT$ that $M=0$ means that all triangles in $ \T$ have at least one vertex in the interior of the domain $\Omega$;
the more general case allows at most $M$ many edge-connected triangles that link interior vertices to all  triangles  in $ \T$.

\subsection*{Proof for $p=3$ for $M=0$} The point is that any triangle $K$  in $\T$ belongs to the patch $T(\vz)\ni K$ of at least
one interior vertex $\vz\in \V(\Omega)$. Hence
 \cite[Thm.~2.1]{Crouzeix_Falk} applies verbatim to the macro-elements $(\omega(\vz):\vz\in\V(\Omega))$
 %, with the set $\V(\Omega)$ of interior vertices,
 and shows that  \eqref{eqfinalsection1asH1}$\equiv$H0 implies Theorem~\ref{thm1}. It moreover provides an
 explicit bound for $\beta_3$ in terms of $\gamma$ and the shape regularity of $\T$ (and a bounded overlap of $L=3$ of the macro-elements).
\qed

\subsection*{Proof for $p=3$ for $M\ge 1$} The macro-elements have to be enlarged in this case to include all triangles; for instance to $\Omega(\vz)$ covered by
the (regular) sub-triangulation $\widehat{\T(\vz)} $ that consists of $\T(\vz)$ plus $M$ layers of
triangles in $\T$ around it. The proof of H0 for $\Omega(\vz)$ requires
Lemma~6.1 in \cite{Crouzeix_Falk}, which is  reformulated as follows for $p=3$.

\begin{lemma}[Crouzeix-Falk]\label{lemmaCrouzeix-Falk}
Suppose the interior edge $E=\partial T_+\cap \partial T_-$ is the common edge of the triangles $T_+,T_-\in\T$ with $\T(E):=\{T_+,T_-\}$ %denote the
%triangulation  of the edge patch
and $\omega(E):=\text{\rm int}(T_+\cup T_-)$. % into the two neighboring triangles.
Then the piecewise divergence
$\ddiv_\pw:$ $ \CR^p_0(\T(E)$; $\R^2)\to P_{p-1}(\T(E))/\R $ has a right inverse in $P_{p-1}(T_+)$
(the full set of polynomials of degree at most $p-1$ in $T_+$) in the sense that
\begin{align*}
\forall g\in \P_{p-1}(T_+) \; & \exists \vv_h\in \CR^p_0(\T(E);\R^2)\quad
g=\ddiv_\pw \vv_h \text{ in }T_+\text{ and }\\
\trinorm{\vv_h}_{\omega(E),\pw }
&:=\sqrt{|\vv_h|_{H^1(\text{\rm int}(T_+))}^2+|\vv_h|_{H^1(\text{\rm int}(T_-))}^2 }\le \|g\|_{L^2(T_+)}/\gamma.
\end{align*}
(The constant $\gamma>0$ depends exclusively on $\epsilon$ and $p\ge 3$.)
% and $\trinorm{\vv_h}_{\omega(\vz),\pw }$ denotes the piecewise $H^1$ norm in $\T(E)$.)
\end{lemma}

The proof of Lemma~\ref{lemmaCrouzeix-Falk} for $p\ge 4$ is postponed
and we continue with its application to the cubic case $p=3$. Arguing as in  Lemma~6.2 in \cite{Crouzeix_Falk},
Lemma~\ref{lemmaCrouzeix-Falk} allows to  enlarge $\omega(\vz)$ (where H0 holds) successively to the bigger domain  $\Omega(\vz) $.
This  verifies  H0 for  the macro-element $\Omega(\vz)$ for each $\vz\in\V(\Omega)$:
The piecewise divergence $ \ddiv_\pw:\CR^p_0(\widehat{\T(\vz)};\R^2)\to\P_{p-1}(\widehat{\T(\vz)})/\R$ is  surjective
with a bounded right inverse in the sense that
\begin{align}\nonumber
\forall g  \in \P_{p-1}&(\widehat{\T(\vz)})/\R \quad \exists \vv_h\in \CR^p_0(\widehat{\T(\vz)};\R^2)\quad
g=\ddiv_\pw \vv_h\text{ in }\Omega(\vz) \text{ and }\\  \label{eqfinalsection1asH1next}
&\trinorm{\vv_h}_{ \Omega(\vz),\pw }
:=\textstyle  \sqrt{\sum_{T\in \widehat{\T(\vz)} }|\vvh|_{H^1(\text{\rm int}(T))}^2    }
\le \|g\|_{L^2(\Omega(\vz))}/\gamma
\end{align}
holds for a constant $\gamma>0$ that exclusively depends  on $\epsilon, M,$ and $p$.
The hypothesis  \eqref{eqfinalsection1asH1next}$\equiv$H0 in \cite{Crouzeix_Falk}
 for the macro-elements $(\Omega(\vz):\vz\in\V(\Omega))$ allows an application  of \cite[Thm.~2.1]{Crouzeix_Falk}
with a larger overlap constant $L$ (that depends on $\epsilon$ and $ M$). This concludes the proof of Theorem~\ref{thm1} for $p=3$. \qed

\subsection*{Proof  of Theorem~\ref{thm1} for $p=5,7,9,\dots $}
The results in Section~2 hold   for any  polynomial degree $p\ge 3$
and lead to the design of $5m$ functions in \eqref{eqcccubicsection1}.
The function $\psi_j$ defined in \eqref{eqcccpsiinubicsection}, however,  does not belong to $\CR^p_0(\T)$ for odd $p\ge 5$ and,
with the notation of Figure~\ref{fig:nodal_patch}, may be replaced by
\begin{equation}\label{eqcccpsiinubicsectionpartplarge}
\psi_j(x):= \frac{6}{\Le_p'(1)} \Bigl( \Le_p( 1-2\varphi_k(x))  +(5\Le_p'(1)-30)\varphi_z^2(x)\varphi_j^2(x)\Bigr)
\end{equation}
at $ x\in T(j)$ { with } $k=j+1$ resp.\  at $x\in T(j-1)$  with $k=j-1 $ and zero elsewhere
($\Le_p'(1)>0$ denotes the derivative of the Legendre polynomial of degree $p$ at $1$);
then $\psi_j$ in \eqref{eqcccpsiinubicsectionpartplarge} coincides with that in \eqref{eqcccpsiinubicsection} for $p=3$.

The important point is that \eqref{eqcccpsiinubicsectionpartplarge} defines a (scalar) Crouzeix-Raviart function $\psi_j\in \CR^p_0(\T)$
for all odd $p\ge 3$  with a scaling such that
$\ddiv ( \psi_j\vn(j))$  is equal to $-12  \vn(j)\cdot\nabla\varphi_k $ at all three vertices of $T(\ell)$ and shifted in the integral along $E(j)$ such that
the integral  of $\ddiv ( \psi_j\vn(j))$ over $T(\ell)$ is equal to $(-1)^{\ell+1-j}|E(j)|$ for the relevant triangles number $\ell=j-1,j$. Those values are
independent of $p=3,5,7,\dots$.
In other words,  for any $T\in\T(\vz)$,  the  functionals in   $\La_T$ from  Section~\ref{sectionEliminationofinteriordegreesoffreedom} evaluated at  $\ddiv(\psi_j\vn(j))$
with  $\psi_j$ from \eqref{eqcccpsiinubicsectionpartplarge}  coincide for all $p=3,5,7,\dots$. As a consequence the Vandermonde matrix $\bM$
in \eqref{eqcccubicsection2matrixM} {\em is the same for all}  $p=3,5,7,\dots$;  whence $\ker \bM=\{\vs\}$.
This and the macro-element methodology  allow  the proof  analogously to the cubic case; hence further details are omitted.

\subsection*{Proof of Lemma~\ref{lemmaCrouzeix-Falk} for $p\ge 4$}
Suppose $E:=E(2)=\partial T_1\cap \partial T_2$ is the common edge of the triangles $T_+:=T_1$ and $T_-:=T_2$ in $\T(E):=\{T_1,T_2\}$ in the notation of
Figure~\ref{fig:nodal_patch}.
Let $\psi_2$ denote the function in \eqref{eqcccpsiinubicsectionpartplarge} for $j=2$ with  the edge $E\equiv E(2)$ with unit normal $ \vn:=\vn(2)$.
The considerations  in the proof of Lemma~\ref{lemma:ccsas3} and in the previous proof  show that the
 $4\times 4$ Vandermonde matrix of  the four functionals in $\La_{T(1)}$ and the four  functions
 \[
12\, \varphi_2\varphi_z^2\nabla\varphi_2|_{T(1)},
 12\, \varphi_2^2\varphi_z\nabla\varphi_z|_{T(1)},
 \psi_2 \vn,  30\, \varphi_2^2\varphi_z^2  \vn \in \CR^p_0(\T(E);\R^2)
\]
reads, for  odd $p\ge 3$ and with the abbreviation $\delta:=-12  \vn\cdot\nabla\varphi_1|_{T(1)}>0$, as
\[
\begin{pmatrix}
12\, |\nabla\varphi_2|_{T(1)}|^2 & 0 & 0 & |E|\, \vn\cdot \nabla\varphi_2|_{T(1)}\\
0&0&12\, |\nabla\varphi_z|_{T(1)}|^2& |E|\,\vn\cdot \nabla\varphi_z|_{T(1)}\\
\delta &(-1)^{p+1} \delta&\delta   &  |E| \\
0&0&0& |E|
\end{pmatrix}.
\]
(For even $p$, $\Le_p'(-1)=-\Le_p'(1)$ leads to a change of signs in the second entry of the third row.)
The $4\times 4$ Vandermonde matrix is regular for all $p\ge 4$.
In other words, the triple formed by the triangle $T$, the vector space of the four functions, and the four functionals in $\La_{T(1)}$ is a finite element in the sense of Ciarlet
and allows for an interpolation. That means, given  any polynomial $q\in P_{p-1}(T(1))$, there exists some
\[
\vv_h\in\text{\rm span}\{  \varphi_2\varphi_z^2\nabla\varphi_2|_{T(1)},
 \varphi_2^2\varphi_z\nabla\varphi_z|_{T(1)},
 \psi_2 \vn,  \varphi_2^2\varphi_z^2  \vn \} %\subset  \CR^p_0(\T(E(2));\R^2)
 \]
with $\La_{T(1)}(q-\ddiv \vv_h)=\mathbf{0}$. Proposition~\ref{propGuzman-Scott} shows $q=\ddiv (\vv_h+ \vb)$  in $T(1)$ for some $\vb\in \B_p(T(1))\subset  \CR^p_0(\T(E);\R^2)$.
Hence the piecewise divergence $\ddiv_\pw:$ $ \CR^p_0(\T(E)$; $\R^2)\to P_{p-1}(\T(E))/\R $ maps surjective onto  $P_{p-1}(T(1))$ and so allows a right-inverse as
asserted. A scaling argument with a contra-variant Piola
transform  (cf., e.g., \cite[p.~79]{Monkbook} or \cite[p.~106]{ErnGuermondFEM1})
shows that $\gamma>0$ depends  on $p$ and  the shape regularity of $\T(E)$.
\qed

 \subsection*{Minimal  Crouzeix-Raviart method}
 This paper has identified Crouzeix-Raviart velocity sub-spaces sufficient for the $\inf$-$\sup$ stability  for $p=3,5,7,\dots$
 as {\em the five functions}  in \eqref{eqcccubicsection1} for any interior edge $E(j)$ plus all the vector-valued bubble-functions $\bB_p(\T)$ in the triangulation $\T$.
 A  practical implementation  for odd $p=3,5,7,\dots$ may therefore utilize  a strict subspace of $\CR^p_0(\T(E);\R^2)$, that consists of
the conforming Lagrange polynomials of degree $p$ in two components and
the scalar $p$-th order  Crouzeix-Raviart edge-bubble function in the normal direction  for each interior edge.

\section*{acknowledgements}
The first author acknowledges the support by the Institute for Mathematical
Research (FIM) at ETH Z\"urich for a research visit during which part of
this work was carried out.

\bibliographystyle{abbrv}
\bibliography{nlailu}
\end{document}

%% file: nodal_patch.tex
\begin{tikzpicture}[%
    line width=0.8pt,%
    scale=2.3,%
    line join=round,%
    line cap=round,%
    >=stealth',%
    rotate=-105]

    \draw (0,0) node[shape=coordinate] (z) {};
    \draw (2,0) node[shape=coordinate] (A1) {A1};
    \draw ($(z)!1!60:(A1)$) node[shape=coordinate] (A2) {A2};
    \draw ($(z)!1.1!45:(A2)$) node[shape=coordinate] (A3) {A3};
    \draw ($(z)!0.8!50:(A3)$) node[shape=coordinate] (A4) {A4};
    \draw ($(z)!1.1!60:(A4)$) node[shape=coordinate] (A5) {A5};
    \draw ($(z)!1!60:(A5)$) node[shape=coordinate] (A6) {A6};

    \draw ($(z)!.5!(A1)$) node[shape=coordinate] (zA1) {zA1};
    \draw ($(z)!.5!(A2)$) node[shape=coordinate] (zA2) {zA2};
    \draw ($(A1)!.5!(A2)$) node[shape=coordinate] (A12) {A12};
    \draw ($(A2)!.5!(A3)$) node[shape=coordinate] (A23) {A23};
    \draw ($(A3)!.5!(A4)$) node[shape=coordinate] (A34) {A34};
    \draw ($(A4)!.5!(A5)$) node[shape=coordinate] (A45) {A45};
    \draw ($(A5)!.5!(A6)$) node[shape=coordinate] (A56) {A56};
    \draw ($(A6)!.5!(A1)$) node[shape=coordinate] (A61) {A61};

    % ANGLES
    \color{black!60}

    \begin{scope}[font=\small]
        % omega
        \pic[draw, angle radius=1.2cm, angle eccentricity=0.7,
        "$\omega(1)$"] {angle = A1--z--A2};
    \pic[draw, angle radius=1.2cm, angle eccentricity=0.7,
        "$\omega(2)$"] {angle = A2--z--A3};
        % \pic[draw, angle radius=.9cm, angle eccentricity=0.7,
        %  "$\omega(3)$"] {angle = A3--z--A4};
        \pic[draw, angle radius=1.2cm, angle eccentricity=0.7,
        "$\omega(j)$"] {angle = A4--z--A5};
        \pic[draw, angle radius=1.2cm, angle eccentricity=0.6,
        "$\omega(m)$"] {angle = A6--z--A1};

        % alpha
        \pic[draw, angle radius=1.4cm, angle eccentricity=0.65,
        "$\alpha(1)$"] {angle = A2--A1--z};
        \pic[draw, angle radius=1.2cm, angle eccentricity=0.65,
        "$\alpha(2)$"] {angle = A3--A2--z};
        % \pic[draw, angle radius=1.2cm, angle eccentricity=0.6,
        % "$\alpha(3)$"] {angle = A4--A3--z};
        \pic[draw, angle radius=1.1cm, angle eccentricity=0.6,
        "$\alpha(j)$"] {angle = A5--A4--z};
        \pic[draw, angle radius=1.4cm, angle eccentricity=0.70,
        "$\alpha(m)$"] {angle = A1--A6--z};

        % beta
        \pic[draw, angle radius=1.2cm, angle eccentricity=0.6,
        "$\beta(1)$"] {angle = z--A2--A1};
        \pic[draw, angle radius=1.2cm, angle eccentricity=0.6,
        "$\beta(2)$"] {angle = z--A3--A2};
    % \pic[draw, angle radius=.9cm, angle eccentricity=0.6,
        %"$\beta(j-1)$"] {angle = z--A4--A3};
        \pic[draw, angle radius=1.2cm, angle eccentricity=0.7,
        "$\beta(j)$"] {angle = z--A5--A4};
        \pic[draw, angle radius=1.4cm, angle eccentricity=0.75,
        "$\beta(m)$"] {angle = z--A1--A6};
    \end{scope}

    % LINES
    \draw (zA1) -- +(90:0.15) arc (90:180:0.15);
    \draw[fill] ($(zA1)+(135:0.08)$) circle[radius=0.2pt];
    \draw (zA2) -- +(150:0.15) arc (150:240:0.15);
    \draw[fill] ($(zA2)+(195:0.08)$) circle[radius=0.2pt];

    % triangles
    \color{black}
    \draw (A1) -- (A2) -- (A3) -- (A4) -- (A5) -- (A6) -- cycle;
    \draw (z) -- (A2);
    \draw (z) -- (A3);
    \draw (z) -- (A1);
    \draw (z) -- (A4);
    \draw (z) -- (A5);
    \draw (z) -- (A6);

    % normals
    \color{black!60}
    \draw[->] (zA1) -- +(90:0.3);
    \draw[->] (zA2) -- +(150:0.3);

    % tangentials
    \draw[->] (zA1) -- +(180:0.3);
    \draw[->] (zA2) -- +(240:0.3);

    % LABELS
    % nodes
    \color{black}
    \draw (z) node[draw,circle,fill=white] {$\vz$};
    \draw ($(z)!1.1!(A1)$) node {$\vP(1)$};
    \draw ($(z)!1.1!(A2)$) node {$\vP(2)$};
    \draw ($(z)!1.1!(A3)$) node {$\vP(3)$};
    \draw ($(z)!1.1!(A4)$) node {$\vP(j)$};
    \draw ($(z)!1.1!(A5)$) node {$\vP(j+1)$};
    \draw ($(z)!1.15!(A6)$) node {$\vP(m)$};

    % triangles
    \draw ($(z)!0.66!(A12)$) node[xshift=10] {$T(1)$};
    \draw ($(z)!0.76!(A23)$) node[yshift=10] {$T(2)$};
  %  \draw ($(z)!0.66!(A34)$) node {$T(3)$};
   \draw ($(z)!0.66!(A34)$) node {$\dots $};
    \draw ($(z)!0.66!(A45)$) node {$T(j)$};
    \draw ($(z)!0.66!(A56)$) node {$\dots$};
    \draw ($(z)!0.66!(A61)$) node {$T(m)$};

    % normals and tangentials
    \draw ($(zA1)+(65:0.3)$) node {$\vn(1)$};
    \draw ($(zA1)+(230:0.25)$) node {$\vt(1)$};

    \draw ($(zA2)+(115:0.35)$) node {$\vn(2)$};
    \draw ($(zA2)+(280:0.3)$) node {$\vt(2)$};

    % edge
    \color{black}
    \draw[line width=0.6pt]
        ($(z)!1.2!(A12)$) node {}
        edge[->,bend right] ($(z)!0.65!(A1)$);
    \draw ($(z)!1.2!(A12)$) node[right]
        {$E(1) = \operatorname{conv}\{\vz, \vP(1)\}$};
\end{tikzpicture}

%% file: ccsasArev_Feb11.bbl
\newcommand{\noopsort}[1]{} \newcommand{\printfirst}[2]{#1}
  \newcommand{\singleletter}[1]{#1} \newcommand{\switchargs}[2]{#2#1}
  \def\cprime{$'$} \def\cprime{$'$} \def\cprime{$'$}
\begin{thebibliography}{10}

\bibitem{Ainsworth_Rankin}
M.~Ainsworth and R.~Rankin.
\newblock Fully computable bounds for the error in nonconforming finite element
  approximations of arbitrary order on triangular elements.
\newblock {\em SIAM J. Numer. Anal.}, 46(6):3207--3232, 2008.

\bibitem{Baran_Stoyan}
{\'A}.~Baran and G.~Stoyan.
\newblock Gauss-{L}egendre elements: a stable, higher order non-conforming
  finite element family.
\newblock {\em Computing}, 79(1):1--21, 2007.

\bibitem{zbMATH03825382}
J.~M. {Boland} and R.~A. {Nicolaides}.
\newblock {Stability of finite elements under divergence constraints}.
\newblock {\em {SIAM J. Numer. Anal.}}, 20:722--731, 1983.

\bibitem{scottbrenner3}
S.~C. Brenner and L.~R. Scott.
\newblock {\em The mathematical theory of finite element methods}, volume~15.
\newblock Springer, New York, third edition, 2008.

\bibitem{CCSS_CR2}
C.~Carstensen and S.~Sauter.
\newblock Critical functions and inf-sup stability of crouzeix-raviart
  elements.
\newblock {\em Computers \& Mathematics with Applications}, 108:12--23, 2022.

\bibitem{ChaLeeLee}
Y.~Cha, M.~Lee, and S.~Lee.
\newblock Stable nonconforming methods for the {S}tokes problem.
\newblock {\em Appl. Math. Comput.}, 114(2-3):155--174, 2000.

\bibitem{ccss_2012}
P.~G. Ciarlet, P.~Ciarlet, S.~A. Sauter, and C.~Simian.
\newblock Intrinsic finite element methods for the computation of fluxes for
  {P}oisson's equation.
\newblock {\em Numer. Math.}, 132(3):433--462, 2016.

\bibitem{CDS}
P.~Ciarle{t, Jr.}, C.~F. Dunkl, and S.~A. Sauter.
\newblock A family of {C}rouzeix-{R}aviart finite elements in 3{D}.
\newblock {\em Anal. Appl. (Singap.)}, 16(5):649--691, 2018.

\bibitem{Crouzeix_Falk}
M.~Crouzeix and R.~Falk.
\newblock Nonconforming finite elements for {S}tokes problems.
\newblock {\em Math. Comp.}, 186:437--456, 1989.

\bibitem{CrouzeixRaviart}
M.~Crouzeix and P.-A. Raviart.
\newblock Conforming and nonconforming finite element methods for solving the
  stationary {S}tokes equations. {I}.
\newblock {\em Rev. Fran\c caise Automat. Informat. Recherche Op\'erationnelle
  S\'er. Rouge}, 7(R-3):33--75, 1973.

\bibitem{ErnGuermondFEM1}
A.~Ern and J.-L. Guermond.
\newblock {\em Finite elements 1}.
\newblock Springer International Publishing, 2021.

\bibitem{Fortin_Soulie}
M.~Fortin and M.~Soulie.
\newblock A nonconforming quadratic finite element on triangles.
\newblock {\em International Journal for Numerical Methods in Engineering},
  19:505--520, 1983.

\bibitem{Greub}
W.~H. Greub.
\newblock {\em Linear algebra}.
\newblock Fourth edition. Graduate texts in mathematics; v 23. Springer-Verlag
  New York, Inc., New York, 1975.

\bibitem{GuzmanScott2019}
J.~Guzm\'{a}n and L.~R. Scott.
\newblock The {S}cott-{V}ogelius finite elements revisited.
\newblock {\em Math. Comp.}, 88(316):515--529, 2019.

\bibitem{hornjohnson}
R.~A. Horn and C.~R. Johnson.
\newblock {\em Matrix analysis}.
\newblock Cambridge University Press, Cambridge, second edition, 2013.

\bibitem{Monkbook}
P.~Monk.
\newblock {\em Finite element methods for {M}axwell's equations}.
\newblock Oxford University Press, New York, 2003.

\bibitem{MR3356019}
M.~Neilan.
\newblock Discrete and conforming smooth de {R}ham complexes in three
  dimensions.
\newblock {\em Math. Comp.}, 84(295):2059--2081, 2015.

\bibitem{ScottVogelius}
L.~R. Scott and M.~Vogelius.
\newblock Norm estimates for a maximal right inverse of the divergence operator
  in spaces of piecewise polynomials.
\newblock {\em RAIRO Mod\'{e}l. Math. Anal. Num\'{e}r.}, 19(1):111--143, 1985.

\bibitem{Stenberg_macro}
R.~Stenberg.
\newblock Analysis of mixed finite elements methods for the {S}tokes problem: a
  unified approach.
\newblock {\em Math. Comp.}, 42(165):9--23, 1984.

\bibitem{BaranCVD}
G.~Stoyan and {\'A}.~Baran.
\newblock Crouzeix-{V}elte decompositions for higher-order finite elements.
\newblock {\em Comput. Math. Appl.}, 51(6-7):967--986, 2006.

\bibitem{GStrang}
G.~Strang.
\newblock The fundamental theorem of linear algebra.
\newblock {\em The American Mathematical Monthly}, 100(9):848--855, 1993.

\bibitem{MR32557}
O.~Taussky.
\newblock A recurring theorem on determinants.
\newblock {\em Amer. Math. Monthly}, 56:672--676, 1949.

\bibitem{MR696548}
M.~Vogelius.
\newblock A right-inverse for the divergence operator in spaces of piecewise
  polynomials. {A}pplication to the {$p$}-version of the finite element method.
\newblock {\em Numer. Math.}, 41(1):19--37, 1983.

\end{thebibliography}
